\documentclass{amsart}
\usepackage{latexsym, amssymb, amsmath, amsthm}
\usepackage[T1]{fontenc}
\usepackage{lmodern}
\usepackage{enumerate}
\usepackage{mathabx}
\usepackage{csquotes}
\usepackage[mathscr]{euscript}
\usepackage{parskip}
\usepackage{thmtools, thm-restate}
\usepackage{bussproofs}
\usepackage{xcolor}

\makeatletter
\newlength\knuthian@fdfive
\def\mathpal@save#1{\let\was@math@style=#1\relax}
\def\utilde#1{\mathpalette\mathpal@save
              {\setbox124=\hbox{$\was@math@style#1$}%
\setbox125=\hbox{$\fam=3\global\knuthian@fdfive=\fontdimen5\font$}
\setbox125=\hbox{$\widetilde{\vrule height 0pt depth 0pt width \wd124}$}%
               \baselineskip=1pt\relax
               \lineskiplimit=\z@\relax
               \lineskip=1pt\relax
               \vtop{\copy124\copy125\vskip -\knuthian@fdfive}}}
\makeatother

\usepackage{pifont}

\declaretheorem[numberwithin=section]{theorem}
\newtheorem{lemma}[theorem]{Lemma}
\newtheorem{corollary}[theorem]{Corollary}
\newtheorem{proposition}[theorem]{Proposition}

\newtheorem{conjecture}[theorem]{Conjecture}
\newtheorem*{claim}{Claim}
\newtheorem*{theorem*}{Theorem}
\newtheorem{hoped}[theorem]{Hoped For Result}

\theoremstyle{definition}
\newtheorem{definition}[theorem]{Definition}
\theoremstyle{remark}
\newtheorem{remark}[theorem]{Remark}

\title{Evitable iterates of the consistency operator}
\author{James Walsh}
\address{Sage School of Philosophy, Cornell University}
\email{jameswalsh@cornell.edu}

\thanks{Thanks to the referee for extensive and helpful suggestions and corrections.}

\begin{document}

\maketitle

\begin{abstract}
    Why are natural theories pre-well-ordered by consistency strength? In previous work, an approach to this question was proposed. This approach was inspired by Martin's Conjecture, one of the most prominent conjectures in recursion theory. Fixing a reasonable subsystem $T$ of arithmetic, the goal was to classify the recursive functions that are monotone with respect to the Lindenbaum algebra of $T$. According to an optimistic conjecture, roughly, every such function must be equivalent to an iterate $\mathsf{Con}_T^\alpha$ of the consistency operator ``in the limit'' within the ultrafilter of sentences that are true in the standard model.
    
    In previous work the author established the first case of this optimistic conjecture; roughly, every recursive monotone function is either as weak as the identity operator in the limit or as strong as $\mathsf{Con}_T$ in the limit. Yet in this note we prove that this optimistic conjecture fails already at the next step; there are recursive monotone functions that are neither as weak as $\mathsf{Con}_T$ in the limit nor as strong as $\mathsf{Con}_T^2$ in the limit. In fact, for every $\alpha$, we produce a function that is cofinally equivalent to $\mathsf{Con}^\alpha_T$ yet cofinally equivalent to $\mathsf{Con}_T$.
\end{abstract}

\section{Introduction}

Why are natural axiomatic theories pre-well-ordered by consistency strength? It is not clear how to answer this question, nor even how to ask it mathematically, since there is no clear mathematical definition of the ``natural'' theories. Yet the informal question remains a well-known conceptual open problem.

In \cite{montalban2019inevitability, walsh2020note} an approach to this problem was proposed. The approach in question was inspired by Martin's approach to an analogous question in recursion theory: Why are the natural Turing degrees well-ordered by Turing reducibility? To state Martin's Conjecture, let's introduce some important notions. First, a function $f:\mathbb{R}\to\mathbb{R}$ is \emph{degree-invariant} if for all reals $x,y$:
$$x\equiv_T y \Longrightarrow f(x) \equiv_T f(y).$$ It seems that by relativizing the definition of a natural Turing degree one always produces a degree-invariant function on the reals. Second, a \emph{Turing cone} is any set of the form $\{x \mid x \geq_T y \}$. Assuming $\mathsf{AD}$, Martin proved that every degree-invariant set of reals (i.e., every set closed under Turing-equivalence) either contains a Turing cone or is disjoint from a Turing cone. This yields a $\{0,1\}$-valued measure on degree-invariant sets; a degree-invariant set has measure 0 if it is disjoint from a Turing cone and measure 1 if it contains a cone. When we say \emph{almost everywhere}, we mean almost everywhere with respect to this measure. Finally, we say that $f\leq_m g$ if $f(x)\leq_Tg(x)$ for almost all $x$. Martin's Conjecture is stated as follows:

\begin{conjecture}[Martin]
Assume $\mathsf{ZF+AD+DC}$.
\begin{enumerate}
    \item If $f:2^\omega \to 2^\omega$ is degree-invariant and it's not the case that $f$ is increasing almost everywhere, then f is constant almost everywhere.
    \item $\leq_m$ pre-well-orders the set of degree-invariant functions that are increasing almost everywhere. If $f$ has $\leq_m$-rank $\alpha$, then $f'$ has $\leq_m$-rank $\alpha+1$, where $f'(x)=f(x)'$ for all $x$.
\end{enumerate}
\end{conjecture}

\emph{Roughly speaking}, Martin's Conjecture says that the only definable degree-invariant functions, up to $\equiv_m$, are the constant functions, the identity function, and the iterates of the Turing jump.\footnote{Of course, the notion of ``definable'' is left vague in what I have written. More precisely, note that Martin's Conjecture is stated under the hypothesis $\mathsf{AD}$. So Martin's Conjecture will apply to Borel functions and to ever more capacious notions of ``definable function'' depending on the large cardinal axioms one assumes.}

The optimistic hope in \cite{montalban2019inevitability, walsh2020note}
 was that an analogue of Martin's Conjecture holds for axiomatic theories.  Let's fix a sound (i.e., true in the standard model) recursively extension $T$ of elementary arithmetic. Say that a recursive function $\mathfrak{g}$ is \emph{extensional} if for all $\varphi$ and $\psi$:
$$ T\vdash \varphi \leftrightarrow \psi \Longrightarrow T\vdash \mathfrak{g}(\varphi) \leftrightarrow \mathfrak{g}(\psi).$$ Shavrukov and Visser \cite{shavrukov2014uniform} proved that there is a recursive extensional density function on the Lindenbaum algebra of $T$, which precludes any direct analogue Martin's Conjecture in this context. Nevertheless, positive results are available if we replace \emph{extensionality} with the stronger condition of \emph{monotonicity}. Say that a recursive function $\mathfrak{g}$ is \emph{monotone} if for all $\varphi$ and $\psi$:
$$ T\vdash \varphi \to \psi \Longrightarrow T\vdash \mathfrak{g}(\varphi) \to \mathfrak{g}(\psi).$$
We will be exclusively concerned with recursive monotone $\mathfrak{g}$ in this paper. We will be primarily interested in truth-preserving operators with this property, such as the consistency operator.

Monotonicity is an analogue of the recursion-theoretic notion \emph{order-preserving}; $f$ is order-preserving if for all reals $x$ and $y$:
$$x\leq_T y \Longrightarrow f(x) \leq_T f(y).$$
Note that Lutz and Siskind \cite{lutz2021results} proved Part 1 of Martin's Conjecture for order-preserving functions and Slaman and Steel \cite{slaman1988definable} proved Part 2 of Martin's Conjecture for order-preserving Borel functions.

G\"{o}del's second incompleteness theorem tells us that the monotone function $$\varphi \mapsto \big(\varphi \wedge \mathsf{Con}_T(\varphi)\big) $$ produces a strictly stronger sentence whenever $\varphi$ is true in the standard model. That is, the consistency operator acts like a \emph{jump} on sound finite extensions of $T$. Just as with the Turing jump, there are iterates of the consistency operator into the effective transfinite.\footnote{Unlike with Turing degrees, however, transfinite iterates of the consistency operator depend on the presentation of the base theory and the well-ordering.} Indeed, \emph{fixing} a suitable elementary presentation $\prec$ of a well-ordering, we may define the iterates of the consistency operator using the fixed point lemma:
$$T\vdash \mathsf{Con}_T^\alpha(\varphi)\leftrightarrow \forall\beta\prec\alpha \; \mathsf{Con}_T\big( \varphi \wedge \mathsf{Con}_T^\beta(\varphi)\big).\footnote{A definition of suitable ordinal notations is given in \textsection \ref{notations}.}$$

In \cite{montalban2019inevitability,walsh2020note} a number of results relating recursive monotone functions to iterates of the consistency operator are established. To state one such theorem, we introduce a definition. Note that whenever we mention sentences being ``true,'' we mean ``true in the standard structure $\mathbb{N}$.''


\begin{definition}
A \emph{cone} is any set $\mathfrak{C}$ such that, for some $\varphi$, $\mathfrak{C}=\{\psi \mid T\vdash \psi \to \varphi \}$. A \emph{true cone} is a cone that contains a true sentence.
\end{definition}

Now here is a precise statement of a theorem from \cite{walsh2020note}:

\begin{theorem}\label{MC_for_con}
Let $\mathfrak{g}$ be recursive and monotone such that, for some $k\in\mathbb{N}$, for all $\varphi$, $\mathfrak{g}(\varphi)$ is $\Pi_k$. Then one of the following holds:
\begin{enumerate}
\item For all $\varphi$ in a true cone,
$T+\varphi \vdash \mathfrak{g}(\varphi) .$
\item For all $\varphi$ in a true cone,
$T+ \varphi + \mathfrak{g}(\varphi) \vdash \mathsf{Con}_T(\varphi).$
\end{enumerate}
\end{theorem}
Roughly, this says that any sufficiently nice operator must either be as weak as the identity operator in the limit or as strong as the consistency operator in the limit. 

We had conjectured that this would be the first step of a classification of recursive monotone operators. Our hope was to prove---along the lines of Martin's conjecture---that any function meeting the hypotheses of the theorem would either be as strong as $\mathsf{Con}_T^\alpha$ in the limit or as weak as $\mathsf{Con}^\beta$ for some $\beta\prec\alpha$ in the limit. That is, we had hoped to prove a result of the following sort:\footnote{Note that Hoped For Result \ref{hoped-for} follows from Conjecture 1.8 in \cite{walsh2020note}.}
\begin{hoped}\label{hoped-for}
Let $\mathfrak{g}$ be recursive and monotone such that, for some $k\in\mathbb{N}$, for all $\varphi$, $\mathfrak{g}(\varphi)$ is $\Pi_k$. Then, for every $\alpha\succ 0$, one of the following holds:
\begin{enumerate}
\item For all $\varphi$ in a true cone, there is a $\beta\prec\alpha$ such that
$T+\varphi +\mathsf{Con}^\beta_T(\varphi)\vdash \mathfrak{g}(\varphi).$
\item For all $\varphi$ in a true cone,
$T+ \varphi + \mathfrak{g}(\varphi) \vdash \mathsf{Con}^\alpha_T(\varphi).$
\end{enumerate}
\end{hoped}

Of course, Theorem \ref{MC_for_con} already covers the $\alpha=1$ case of the hoped for result. Some results from \cite{montalban2019inevitability} provided optimism for the $\alpha\succ 1$ cases. Before stating these results, let's introduce some definitions. We say that \emph{cofinally many} true sentences belong to a set $\mathfrak{A}$ if $\mathfrak{A}$ is cofinal in the ultrafilter of true sentences, i.e., for every true $\varphi$ there is a true $\psi$ such that $T\vdash \psi \to \varphi$ and $\psi\in\mathfrak{A}$. $[\varphi]_T$ is the equivalence class of $\varphi$ modulo $T$-provable equivalence. That is:
$$[\varphi]_T=\{\psi \mid T \vdash \varphi \leftrightarrow \psi\}.$$
If $T\vdash \varphi \to \psi$, let's say that $\varphi$ \emph{$T$-provably implies} $\psi$. If $\varphi$ $T$-provably implies $\psi$ and $\psi$ does not $T$-provably imply $\varphi$ then we say that $\varphi$ \emph{strictly $T$-provably implies} $\psi$.
Here are the positive results from \cite{montalban2019inevitability}:

\begin{theorem}\label{cofinal}
Let $\mathfrak{g}$ be recursive and monotone. Suppose that for all $\varphi$ both of the following hold:
\begin{enumerate}
    \item $\varphi\wedge \mathsf{Con}_T^\alpha(\varphi)\text{ $T$-provably implies } \mathfrak{g}(\varphi)$.
    \item If $[\mathfrak{g}(\varphi)]_T\neq[\bot]_T$, then for all $\beta\prec\alpha$, $\mathfrak{g}(\varphi)\text{ strictly $T$-provably implies } \mathsf{Con}_T^\beta(\varphi)$.
\end{enumerate}
Then for cofinally many true $\varphi$, $[\mathfrak{g}(\varphi)]_T=[\varphi\wedge\mathsf{Con}_T^\alpha(\varphi)]_T$.
\end{theorem}

This result says that if the range of a recursive monotone $\mathfrak{g}$ is sufficiently constrained, then $\mathfrak{g}$ must coincide \emph{cofinally} with an iterate of the consistency operator. In particular, if $\mathfrak{g}$ is everywhere as weak as $\mathsf{Con}_T^\alpha$ but everywhere strictly stronger than $\mathsf{Con}_T^\beta$ for all $\beta\prec\alpha$, then $\mathfrak{g}$ must coincide \emph{cofinally} with $\mathsf{Con}_T^\alpha$. As a corollary we infer that:

\begin{corollary}
There is no recursive monotone $\mathfrak{g}$ such that for all $\varphi$ such that $[\varphi\wedge\mathsf{Con}_T^\alpha(\varphi)]_T\neq[\bot]_T$, both of the following hold:
\begin{enumerate}
    \item $\varphi\wedge\mathsf{Con}_T^\alpha(\varphi) \text{ strictly $T$-provably implies } \mathfrak{g}(\varphi)$.
    \item For all $\beta\prec\alpha$, $ \mathfrak{g}(\varphi)\text{ strictly $T$-provably implies } \varphi \wedge \mathsf{Con}^\beta(\varphi)$.
\end{enumerate}
\end{corollary}

This is just to say that there is no recursive monotone function of \emph{strictly intermediate} strength, i.e., a function that is everywhere strictly stronger than $\mathsf{Con}^\beta$ for all $\beta\prec\alpha$ yet strictly weaker than $\mathsf{Con}_T^\alpha$. This rules out the most obvious sort of counter-example one might expect to Hoped For Result \ref{hoped-for}.

Nevertheless, in this paper we will see that this hoped for result fails. In fact, it fails already at the very next step, i.e., $\alpha=2$. For any $\alpha\succ 0$, we can construct a function that is cofinally as strong than $\mathsf{Con}^\alpha_T$ yet cofinally as weak as $\mathsf{Con}_T$: 
\begin{theorem}\label{main-thm}
For every $\alpha\succ 0$, there is a recursive monotone $\mathfrak{g}$ such that, for all $\varphi$, $\mathfrak{g}(\varphi)$ is $\Pi_1$, and both of the following hold:
\begin{enumerate}
\item For cofinally many true $\varphi$:
$$[\varphi \wedge \mathfrak{g}(\varphi)]_T = [\varphi \wedge \mathsf{Con}^\alpha_T(\varphi)]_T.$$
\item For cofinally many true $\varphi$:
$$[\varphi \wedge\mathsf{Con}_T(\varphi)]_T= [\varphi \wedge \mathfrak{g}(\varphi)]_T.$$
\end{enumerate}
\end{theorem}
This the most dramatic possible failure of the Hoped For Result \ref{hoped-for} that is consistent with Theorem \ref{MC_for_con}.

Of course, this result is somewhat \emph{negative} considering the context in which it is proved. That is, it shows that the hoped for analogue of Martin's Conjecture fails. Yet this result also highlights something rather surprising about the consistency operator. Indeed, the consistency operator stands apart from its iterates with respect to its inevitability.

Here is our plan for the rest of the paper. In \textsection \ref{preliminaries}, we will cover some preliminaries. In \textsection \ref{sets} we will cover the main technical aspect of our result, which is the construction of pathological recursively enumerable sets. In particular, these sets contain cofinally many true sentences but do not contain any true cones. In \textsection \ref{proof} we will present the proof of the main theorem. In \textsection \ref{observations} we will conclude with some observations; in particular, we will discuss the relationship between the negative Theorem \ref{main-thm} and the positive Theorem \ref{MC_for_con}.


\section{Preliminaries}\label{preliminaries}

In this section we will fix some notation and cover some preliminaries. The main goal of this section is to introduce the iterations of the consistency operator (relative to a fixed ordinal notation system) and prove that they are monotone.

\subsection{Base Theory}

The theories we will be interested in are extensions of elementary arithmetic or $\mathsf{EA}$. The signature of $\mathsf{EA}$ is the usual signature of $\mathsf{PA}$ with a function symbol for $2^x$. $\mathsf{EA}$ has as axioms the $\mathsf{PA}^-$ axioms plus induction for all formulas bounded in an exponential term. $\mathsf{EA}$ is just strong enough to carry out the standard arithmetization of syntax in the usual manner. For details about $\mathsf{EA}$ see \cite{beklemishev2005reflection}.

One of the crucial features of $\mathsf{EA}$ is $\Sigma_1$-completeness, which we will use. This is just to say that for any $T$ extending $\mathsf{EA}$ and any $\Sigma_1$ sentence $\varphi$:
$$\mathbb{N}\vDash \varphi \Rightarrow T\vdash \varphi.$$

From here on out we will fix a sound, elementarily axiomatized extension $T$ of $\mathsf{EA}$. By saying that $T$ is sound we mean that for all $\varphi$:
$$T\vdash \varphi \Rightarrow \mathbb{N}\vDash \varphi.$$
Of course, this means that for every $\Sigma_1$ sentence $\varphi$ we have:
$$\mathbb{N}\vDash \varphi \Leftrightarrow T\vdash \varphi.$$

\subsection{Ordinal Notations}\label{notations}

There are many ways of reasoning with ordinal notations in elementary arithmetic. We will not need to work with ordinal notation systems in much detail; much of what we do can be done given any ``reasonable'' choice. For present purposes, it is enough to briefly mention a few properties our ordinal notations will have. We will call our presentations \emph{suitable} presentations.


Every suitable ordinal notation system is a pair $(D,\prec)$ of elementary formulas, such that:
\begin{enumerate}
\item the relation $\prec$ well-orders $D$ in the standard model of arithmetic;
\item $D$ is provably closed under successor;
\item $\mathsf{EA}$ proves that $\prec$ linearly orders the elements satisfying $D$;\footnote{The referee has pointed out that provable linearity is not required; rather, we only need that the relation is $T$-provably transitive and irreflexive.}
\item the elementary formulas defining the initial ordinal 0 (which need not be the natural number 0), the set of limit ordinals, and the successor relation provably in $\mathsf{EA}$ satisfy their corresponding first order definitions in terms of $\prec$.
\end{enumerate}

\subsection{Iterated Consistency}

Fixing a suitable elementary presentation $\prec$ of a well-ordering, we may define the iterates of the consistency operator using the fixed point lemma. Technically, we find a formula $\mathbf{Con}^\star$ in two variables:
$$T\vdash \mathbf{Con}^\star(\alpha,\varphi)\leftrightarrow \forall\beta\prec\alpha \; \mathsf{Con}_T\big( \varphi \wedge \mathbf{Con}^\star_T(\beta,\varphi)\big).$$
We write $\mathsf{Con}_T^\alpha(\varphi)$ as an abbreviation for $\mathbf{Con}^\star(\alpha,\varphi)$.

\begin{remark}
We warn the reader that there is some discrepancy between our notation and the notation used by other authors. Our iteration scheme $$\mathsf{Con}_T^{\alpha+1}(\varphi)\equiv \mathsf{Con}_T\big( \varphi \wedge \mathsf{Con}_T^\alpha(\varphi)\big)$$ is sometimes denoted $\mathsf{Con}(T+\varphi)_\alpha$; see, e.g., \cite{beklemishev1995iterated}. Moreover, the notation $\mathsf{Con}_T^{\alpha+1}(\varphi)$ is sometimes used to denote $\mathsf{Con}_T\big(\mathsf{Con}^\alpha_T(\varphi))$; see, e.g., \cite{beklemishev1991provability}.
\end{remark}

\begin{remark}
Note that all of the iterates of the consistency operator we have defined are \emph{sentences}. This is appropriate for the current investigation since we are interested in functions on the Lindenbaum algebra of $T$. For instance, given our definition, for a limit $\lambda$, the sentence $\mathsf{Con}^\lambda_T(\varphi)$ says $\forall \alpha\prec \lambda \; \mathsf{Con}_T\big( \varphi \wedge \mathsf{Con}_T^\alpha(\varphi) \big)$. This clashes with the conventions adopted in some other papers, wherein $\mathsf{Con}_T^\lambda(\varphi)$ is used as a name for the infinite set $\{ \mathsf{Con}_T^\alpha(\varphi) \mid \alpha \prec \lambda \}$.
\end{remark}

Note that:\footnote{Here we use the fourth clause of our discussion of suitable ordinal notation systems.}
\begin{flalign*}
T &\vdash \mathsf{Con}^0_T(\varphi)\leftrightarrow \top\\
T &\vdash \mathsf{Con}^1_T(\varphi)\leftrightarrow \mathsf{Con}_T(\varphi)
\end{flalign*}
Since suitable ordinal notations define well-orderings over the standard structure $\mathbb{N}$, it follows by induction that for true $\varphi$, $\mathsf{Con}_T^\alpha(\varphi)$ is also true. Thus, for true $\varphi$ the hierarchy $\mathsf{Con}_T^\alpha(\varphi)$ is proper by G\"{o}del's second incompleteness theorem.

It is immediate from the definition that, for any $\varphi$, the sentence $\mathsf{Con}_T^\alpha(\varphi)$ is $\Pi_1$.

To verify that each function $\varphi \mapsto \mathsf{Con}^\alpha_T(\varphi)$ is monotone, we rely on Schmerl's \cite{schmerl1979fine} technique of reflexive induction. Reflexive induction is a way of simulating large amounts of transfinite induction in weak theories. It is particularly useful for proving claims about iterated reflection principles. The technique is facilitated by the following theorem; we include the proof, since it is so short.

\begin{theorem}[Schmerl]\label{schmerlRI}
Let $T$ be a recursively axiomatized theory containing $\mathsf{EA}$. Suppose $T \vdash \forall \alpha \Big( \mathsf{Pr}_T\big(\forall \beta\prec \alpha \; A(\beta)\big) \to A(\alpha) \Big).$
Then $T\vdash \forall \alpha \; A(\alpha)$.\footnote{Schmerl proves his result using the base theory $\mathsf{PRA}$. Beklemishev has shown that $\mathsf{EA}$ suffices.}
\end{theorem}

\begin{proof}
Suppose that $T \vdash \forall \alpha \Big( \mathsf{Pr}_T\big(\forall \beta\prec \alpha \; A(\beta)\big) \to A(\alpha) \Big).$ We infer that:
\begin{flalign*}
T \vdash \mathsf{Pr}_T\big(\forall \alpha \; A(\alpha)\big) &\to \forall \alpha \mathsf{Pr}_T\big( \forall\beta\prec\alpha \; A(\alpha) \big) \\
& \to \forall \alpha A(\alpha).
\end{flalign*}

L\"{o}b's Theorem then yields $T\vdash \forall \alpha \; A(\alpha)$.
\end{proof}

We now will verify that the function $\varphi \mapsto\mathsf{Con}_T^\alpha(\varphi)$ is monotone.

\begin{lemma}\label{first-mono}
For any $\varphi$ and $\psi$, if $T\vdash \varphi \to \psi$ then $T\vdash \forall \alpha \big( \mathsf{Con}_T^\alpha(\varphi) \to \mathsf{Con}_T^\alpha(\psi) \big)$.
\end{lemma}

\begin{proof}
Suppose that:
\begin{equation}\label{first}
    T\vdash \varphi \to \psi.
\end{equation}
We prove the claim by reflexive induction.

\emph{Reason in $T$:} Let $\alpha$ be arbitrary. Assume the reflexive induction hypothesis:
\begin{equation}\label{second}
 \mathsf{Pr}_T\Big(  \forall\beta\prec\alpha \; \big( \mathsf{Con}_T^\beta(\varphi) \to \mathsf{Con}_T^\beta(\psi) \big) \Big).
\end{equation}

We reason as follows:
\begin{flalign*}
\mathsf{Con}_T^\alpha(\varphi) &\to \forall\beta\prec\alpha \; \mathsf{Con}_T\big(\varphi \wedge \mathsf{Con}_T^\beta(\varphi)\big) \text{ by definition of $\mathsf{Con}_T^\alpha$;}\\
&\to \forall\beta\prec\alpha \; \mathsf{Con}_T\big(\psi \wedge \mathsf{Con}_T^\beta(\varphi)\big) \text{ by \ref{first} and monotonicity of $\mathsf{Con}_T$;}\\
&\to \forall\beta\prec\alpha \; \mathsf{Con}_T\big(\psi \wedge \mathsf{Con}_T^\beta(\psi)\big) \text{ by \ref{second};}\\
&\to  \mathsf{Con}_T^\alpha(\psi) \text{ by definition of $\mathsf{Con}^\alpha_T$.}
\end{flalign*}

\emph{Reasoning externally now:} Let $A(\gamma)$ denote the claim: $$\mathsf{Con}^\gamma_T(\varphi)\to\mathsf{Con}^\gamma_T(\psi) \big.$$

We have shown that:
$$T\vdash \forall \alpha \Big( \mathsf{Pr}_T\big(\forall\beta\prec \alpha \; A(\beta)\big) \to A(\alpha) \Big).$$
By applying Lemma \ref{schmerlRI}, we infer that $T\vdash \forall \alpha \; A(\alpha)$.
\end{proof}

We want to see not only that $\alpha$-iterated consistency is monotone but also that it is provably monotone in $T$. We carry out this argument in two steps.

\begin{corollary}\label{second-mono}
$T\vdash \forall \varphi \forall \psi \Bigg( \mathsf{Pr}_T(\varphi\to\psi)\to \mathsf{Pr}_T\Big( \forall \alpha \big(\mathsf{Con}_T^\alpha(\varphi) \to \mathsf{Con}_T^\alpha(\psi)\big)\Big) \Bigg).$
\end{corollary}

\begin{proof}
Since the proof of Lemma \ref{first-mono} can be carried out in $T$. 
\end{proof}

\begin{corollary}\label{third-mono}
$T\vdash \forall \varphi \forall \psi \Big( \mathsf{Pr}_T(\varphi\to\psi)\to  \forall \alpha \big(\mathsf{Con}_T^\alpha(\varphi) \to \mathsf{Con}_T^\alpha(\psi)\big) \Big).$
\end{corollary}

\begin{proof}
\emph{Reason in $T$:} Let $\varphi$ and $\psi$ be arbitrary. Suppose that:
\begin{equation}\label{implication-hyp}
    T\vdash \varphi \to \psi.
\end{equation}
Let $\alpha$ be arbitrary. We reason as follows:
\begin{flalign*}
\mathsf{Con}_T^\alpha(\varphi) &\to \forall\beta\prec\alpha \; \mathsf{Con}_T\big(\varphi \wedge \mathsf{Con}_T^\beta(\varphi)\big) \text{ by definition of $\mathsf{Con}_T^\alpha$;}\\
&\to \forall\beta\prec\alpha \; \mathsf{Con}_T\big(\psi \wedge \mathsf{Con}_T^\beta(\varphi)\big) \text{ by \ref{implication-hyp} and monotonicity of $\mathsf{Con}_T$;}\\
&\to \forall\beta\prec\alpha \; \mathsf{Con}_T\big(\psi \wedge \mathsf{Con}_T^\beta(\psi)\big) \text{ by Corollary \ref{second-mono};}\\
&\to  \mathsf{Con}_T^\alpha(\psi) \text{ by definition of $\mathsf{Con}^\alpha_T$.}
\end{flalign*}
This completes the proof of the corollary.
\end{proof}


\section{Constructing Pathological Sets}\label{sets}

The main technical aspect of our result is the construction of recursively enumerable sets that contain arbitrarily strong true sentences but that have wide gaps. In this section (and the next) we will work with a fixed suitable ordinal notation system $\prec$; see \textsection \ref{notations}. We will define a set $\mathfrak{A}_\alpha$ for each ordinal notation $\alpha\succeq 0$; the size of the gaps that we leave in the set will depend on $\alpha$. 

Similar constructions of recursively enumerable sets appear in \cite{montalban2019inevitability, walsh2020note}. The goal in the construction of these sets was merely to include arbitrarily strong true sentences and to omit arbitrarily strong true sentences. These sets did not leave large enough gaps for present purposes.

Here is how we define the set $\mathfrak{A}_\alpha$:

Let $\varphi_0,\varphi_1,\dots,$ be an effective G\"{o}del numbering of arithmetical sentences. We describe the construction of $\mathfrak{A}_\alpha$ in stages. During a stage $n$ we may activate finitely many sentences; if $\psi$ is some such sentence we say that $\psi$ is \emph{active} until $\psi$ is deactivated at the later stage $n+1$.

\textbf{Stage 0:} Numerate $\top$ into $\mathfrak{A}_\alpha$; activate $\top\wedge\mathsf{Con}_T^{\alpha+1}(\top)$.


\textbf{Stage n+1:} There are finitely many active sentences. For each active sentence $\psi$, numerate $\theta_0:= \psi \wedge \varphi_{n}$ and $\theta_1:= \psi \wedge \neg \varphi_{n}$ into $\mathfrak{A}_\alpha$. Deactivate the sentence $\psi$ and activate the sentences $\theta_0\wedge \mathsf{Con}_T^{\alpha+1}(\theta_0)$ and $\theta_1\wedge \mathsf{Con}_T^{\alpha+1}(\theta_1)$.


It can be useful to visualize, along with the construction of $\mathfrak{A}_\alpha$, the construction of a tree that is binary branching. The tree has $\top$ as its root. The nodes in the tree are the sentences that are numerated into $\mathfrak{A}_\alpha$. Informally, the immediate descendants of any sentence $\psi$ are the sentences that are numerated into $\mathfrak{A}_\alpha$ immediately after $\psi$ on account of $\psi$. More formally, for any sentence $\psi$ numerated into $\mathfrak{A}_\alpha$ at stage $n$, say that $\psi < \psi \wedge \mathsf{Con}_T^{\alpha+1}(\psi) \wedge \pm \varphi_{n}$; the tree ordering is the transitive closure of the ordering $<$ (note that this ordering is defined only on sentences numerated into $\mathfrak{A}_\alpha$).


We will call the branches through this tree \emph{$\mathfrak{A}_\alpha$-branches}. If $\varphi$ and $\psi$ share an $\mathfrak{A}_\alpha$-branch and $\varphi$ was numerated into $\mathfrak{A}_\alpha$ at stage $n$ and $\psi$ was numerated into $\mathfrak{A}_\alpha$ at stage $k$ where $k> n$, we say that $\psi$ is a \emph{descendant} of $\varphi$. If $k=n+1$ we say that $\psi$ is an immediate descendant of $\varphi$.

\begin{remark}\label{con-imply}
Note that for each ordinal notation $\alpha$ we get a set $\mathfrak{A}_\alpha$. The gaps that we leave in $\mathfrak{A}_\alpha$ depend on $\alpha$ in the following sense: If $\varphi$ is numerated into $\mathfrak{A}_\alpha$, then $\varphi$'s immediate descendants imply $\mathsf{Con}^{\alpha+1}_T(\varphi)$.
\end{remark}

We can easily check some basic properties of this set $\mathfrak{A}_\alpha$.

\begin{lemma}
$\mathfrak{A}_\alpha$ is recursively enumerable.
\end{lemma}

\begin{proof}
By construction.
\end{proof}

\begin{remark}
An important consequence of the recursive enumerability of $\mathfrak{A}_\alpha$ is that for any $\varphi$, if $\varphi\in\mathfrak{A}_\alpha$ then $T\vdash \varphi\in\mathfrak{A}_\alpha$. This follows since $T$ is $\Sigma_1$-complete.
\end{remark}

\begin{lemma}
$\mathfrak{A}_\alpha$ contains arbitrarily strong true sentences.
\end{lemma}

\begin{proof}
Let $\varphi_n$ be a true sentence. By induction it is easy to see that there is one true active sentence at each stage. Let's say that going into stage $n+1$ the true active sentence is $\psi$. Then at stage $n+1$ we numerate $\psi\wedge \varphi_n$ into $\mathfrak{A}_\alpha$.
\end{proof}

$\mathfrak{A}_\alpha$-branches are sets of formulas, so reasoning about $\mathfrak{A}_\alpha$-branches might seem to require second-order expressive resources. Yet we have assumed only that $T$ contains elementary arithmetic. Nevertheless, elementary arithmetic suffices for reasoning about the descendant relation. The claims we make about $\mathfrak{A}_\alpha$-branches in this paper could be translated into $T$-intelligible claims about the descendant relation. In the following lemma, for instance, we will prove a claim \emph{within $T$} about $\mathfrak{A}_\alpha$-branches. All such claims can be translated into claims about the descendant relation, though we will not give an explicit translation here.

\begin{lemma}\label{branches}
Provably in $T$, if $\psi$ and $\theta$ both belong to $\mathfrak{A}_\alpha$ but do not share an $\mathfrak{A}_\alpha$-branch, then $\psi$ and $\theta$ are jointly $T$-inconsistent.
\end{lemma}

\begin{proof}
\emph{Reason in $T$:} First observe that for any two sentences $\varphi$ and $\psi$ in the tree, if $\varphi$ is a descendant of $\psi$ then $T+\varphi\vdash\psi$. 

Now let $\psi$ and $\theta$ be arbitrary sentences in $\mathfrak{A}_\alpha$ that do not share an $\mathfrak{A}_\alpha$-branch. Then there is some node $\zeta_0$ that has immediate descendants $$\zeta_1:=\zeta_0 \wedge \mathsf{Con}_T^{\alpha+1}(\zeta_0)\wedge \varphi_n$$ and $$\zeta_2:=\zeta_0 \wedge \mathsf{Con}_T^{\alpha+1}(\zeta_0)\wedge \neg \varphi_n$$ such that $T+\psi\vdash\zeta_1$ and $T+\theta\vdash \zeta_2$. But $\zeta_1$ and $\zeta_2$ are jointly inconsistent whence $\psi$ and $\theta$ are too.
\end{proof}



\begin{lemma}\label{inconsistent}
Provably in $T$, some $T$-refutable sentence $\theta$ belongs to $\mathfrak{A}_\alpha$.
\end{lemma}

\begin{proof}
Let $T\vdash \neg\psi$. Note that $\psi$ is $\varphi_n$ for some $n$. At stage $n$, we numerate a sentence $\theta$ that $T$-provably implies $\psi$ into $\mathfrak{A}_\alpha$. Note that $T$ proves both that $T\vdash \neg\theta$ and that $\theta\in\mathfrak{A}_\alpha$, since $T$ is $\Sigma_1$-complete.
\end{proof}



\section{The Proof}\label{proof}

Now we are ready to prove the main theorem. In this section we provide an example of a recursive monotone function that oscillates between behaving like $\mathsf{Con}_T^\alpha$ and behaving like $\mathsf{Con}_T$. Note that this refutes the optimistic Hoped For Result \ref{hoped-for}. Indeed, no function that oscillates cofinally between behaving like $\mathsf{Con}_T^\alpha$ and behaving like $\mathsf{Con}_T$ can converge on either in the limit (assuming that $\alpha\succ 1$).


For convenience, we restate Theorem \ref{main-thm} here.

\begin{theorem}
For every $\alpha\succ 0$, there is a recursive monotone $\mathfrak{g}$ such that, for all $\varphi$, $\mathfrak{g}(\varphi)$ is $\Pi_1$, and both of the following hold:
\begin{enumerate}
\item For cofinally many true $\varphi$:
$$[\varphi \wedge \mathfrak{g}(\varphi)]_T = [\varphi \wedge \mathsf{Con}^\alpha_T(\varphi)]_T.$$
\item For cofinally many true $\varphi$:
$$[\varphi \wedge\mathsf{Con}_T(\varphi)]_T= [\varphi \wedge \mathfrak{g}(\varphi)]_T.$$
\end{enumerate}
\end{theorem}

\begin{proof}
Given $\alpha\succ 0$, let:
$$\mathfrak{g}(\varphi):=\forall\theta\in\mathfrak{A}_\alpha\big( \mathsf{Pr}_T(\varphi \to \theta) \to \mathsf{Con}_T^\alpha(\theta) \big).$$

Note that $\mathfrak{g}$ is clearly recursive. It is routine to check that, for all $\varphi$, $\mathfrak{g}(\varphi)\in\Pi_1$ and (using Corollary \ref{third-mono}) that $\mathfrak{g}$ is monotone.  

Let $\varphi\in\mathfrak{A}_\alpha$. By $\Sigma_1$-completeness of $T$:
\begin{equation}\label{sigma}
    T\vdash \varphi \in\mathfrak{A}_\alpha. \tag{$\triangle$}
\end{equation}

We reason as follows:
\begin{flalign*}
T+\mathfrak{g}(\varphi) &\vdash \forall\theta\in\mathfrak{A}_\alpha\big( \mathsf{Pr}_T(\varphi \to \theta) \to \mathsf{Con}_T^\alpha(\theta) \big) \text{ by choice of $\mathfrak{g}$};\\
T +\mathfrak{g}(\varphi) &\vdash  \mathsf{Pr}_T(\varphi \to \varphi) \to \mathsf{Con}_T^\alpha(\varphi) \text{ by ($\triangle$);}\\
T +\mathfrak{g}(\varphi) &\vdash   \mathsf{Con}_T^\alpha(\varphi) \text{ by $\Sigma_1$-completeness.}
\end{flalign*}

On the other hand, the monotonicity of $\mathsf{Con}^\alpha_T$ is provable in $T$ (see Corollary \ref{third-mono}). Whence:
\begin{flalign*}
T+\mathsf{Con}^\alpha_T(\varphi) &\vdash \forall\theta\big( \mathsf{Pr}_T(\varphi \to \theta) \to \mathsf{Con}_T^\alpha(\theta) \big);\\
T+\mathsf{Con}^\alpha_T(\varphi) &\vdash \mathfrak{g}(\varphi).
\end{flalign*}

Since cofinally many true sentences belong to $\mathfrak{A}_\alpha$, this already takes care of (1).



Now we pick some true $\psi\in \mathfrak{A}_\alpha$. We consider the sentence $\varphi:=\psi\wedge\mathsf{Con}^\alpha_T(\psi)$.

\begin{claim}
$T$ proves that if $\varphi$ is consistent, then $\psi$ is the strongest sentence in $\mathfrak{A}_\alpha$ that $\varphi$ $T$-provably implies.
\end{claim}

To see that the claim is true, \emph{we reason in $T$:} Suppose that $\varphi$ is consistent. Note that $\varphi$ $T$-provably implies $\psi$. Note, moreover, that $\psi$ is inconsistent with every sentence in $\mathfrak{A}_\alpha$ with which $\psi$ does not share an $\mathfrak{A}_\alpha$-branch by Lemma \ref{branches}. So, since $\varphi$ is consistent, the only $\mathfrak{A}_\alpha$ sentences that $\varphi$ $T$-provably implies must share an $\mathfrak{A}_\alpha$-branch with $\psi$. By construction of $\mathfrak{A}_\alpha$, every descendant of $\psi$ $T$-provably implies $\mathsf{Con}^{\alpha+1}_T(\psi)$. But, since $\varphi$ is consistent, $\varphi$ does not $T$-provably imply $\mathsf{Con}^{\alpha+1}_T(\psi)$ by G\"{o}del's second incompleteness theorem. This delivers the claim. 

We then reason as follows:
\begin{flalign*}
T+\varphi +\mathsf{Con}_T(\varphi)&\vdash \forall \theta\in\mathfrak{A}_\alpha\Big( \mathsf{Pr}_T(\varphi\to\theta) \to \mathsf{Pr}_T(\psi\to\theta) \Big) \text{ by the claim;}\\
T+\varphi +\mathsf{Con}_T(\varphi)&\vdash \forall \theta\in\mathfrak{A}_\alpha\Big( \mathsf{Pr}_T(\varphi\to\theta) \to \big(\mathsf{Con}^\alpha_T(\psi) \to \mathsf{Con}^\alpha_T(\theta) \big)\Big) \text{ by Corollary \ref{third-mono}};\\
T+\varphi +\mathsf{Con}_T(\varphi)&\vdash \forall \theta\in\mathfrak{A}_\alpha\Big( \mathsf{Pr}_T(\varphi\to\theta) \to \mathsf{Con}^\alpha_T(\theta) \Big) \text{ since $T+\varphi\vdash \mathsf{Con}_T^\alpha(\psi)$;}\\
T+\varphi +\mathsf{Con}_T(\varphi)&\vdash \mathfrak{g}(\varphi) \text{ by the definition of $\mathfrak{g}$.}
\end{flalign*}

For the converse:
\begin{flalign*}
T&\vdash \exists \theta \in\mathfrak{A}_\alpha   \neg\mathsf{Con}_T(\theta) \text{ by Lemma \ref{inconsistent};}\\
T+\neg \mathsf{Con}_T(\varphi)&\vdash \exists \theta \in\mathfrak{A}_\alpha \big( \mathsf{Pr}_T(\varphi\to\theta) \wedge \neg\mathsf{Con}_T^\alpha(\theta) \big) ;\\
T+\neg \mathsf{Con}_T(\varphi)&\vdash \neg\mathfrak{g}(\varphi)\text{ by the definition of $\mathfrak{g}$.}
\end{flalign*}

This takes care of (2).
\end{proof}

\section{Observations}\label{observations}

Theorem \ref{main-thm} refutes Hoped For Result \ref{hoped-for}. Yet what happens in the case $\alpha=0$? That is, why can't the proof of Theorem \ref{main-thm} be adapted to the $\alpha=0$ case, thereby contradicting Theorem \ref{MC_for_con}? The answer exhibits an important feature that the notion of consistency does not share with its iterates.

Recall that we construct $\mathfrak{A}_\alpha$ so that whenever $\varphi$ is numerated into $\mathfrak{A}_\alpha$, then $\varphi$'s immediate descendants $T$-provably imply $\mathsf{Con}^{\alpha+1}_T(\varphi)$. In particular, whenever $\varphi$ is numerated into $\mathfrak{A}_0$, then $\varphi$'s immediate descendants $T$-provably imply $\mathsf{Con}_T(\varphi)$. Let's consider the function:
$$\mathfrak{g}_0(\varphi):= \forall \theta \in \mathfrak{A}_0\big( \mathsf{Pr}_T(\varphi \to \theta) \to \mathsf{Con}_T(\theta) \big) .$$


Surprisingly, $\mathfrak{g}_0$ is actually equivalent to the consistency operator. That is:

\begin{proposition}\label{avoid}
For every $\varphi$, $T\vdash \mathsf{Con}_T(\varphi) \leftrightarrow\mathfrak{g}_0(\varphi).$
\end{proposition}

\begin{proof}
Left to right:
\begin{flalign*}
T \vdash \mathsf{Con}_T(\varphi) &\to \forall \theta\big( \mathsf{Pr}_T(\varphi \to \theta) \to \mathsf{Con}_T(\theta)\big);\\
& \to \mathfrak{g}_0(\varphi).
\end{flalign*}

Right to left:
\begin{flalign*}
T \vdash \neg \mathsf{Con}_T(\varphi) & \to \exists \theta \in\mathfrak{A}_0 \big( \mathsf{Pr}_T(\varphi\to\theta) \wedge \neg\mathsf{Con}_T(\theta) \big) \text{ by Lemma \ref{inconsistent}};\\
&\to \neg \mathfrak{g}_0(\varphi).
\end{flalign*}
This completes the proof.
\end{proof}
Note the appeal to Lemma \ref{inconsistent}. Here we use that $\mathfrak{A}_0$ is guaranteed to contain an inconsistent sentence; the important point is that if $\varphi$ is $T$-inconsistent, then some sentence that $\varphi$ $T$-provably implies belongs to  $\mathfrak{A}_0$. Indeed, the fact that all $T$-inconsistent sentences $T$-provably imply each other is used in the proof of the positive Theorem \ref{MC_for_con} (see the proof of Theorem 2.4 Case 2 in \cite{walsh2020note}). By contrast, if we merely knew $\neg\mathsf{Con}_T^{\alpha+1}(\varphi)$, we would not be able to conclude that some $T$-consequence of $\varphi$ belongs to  $\mathfrak{A}_\alpha$. Nor for any of the other iterates of the consistency operator.

There are ways of modifying $\mathfrak{g}_0$ to avoid Proposition \ref{avoid}. Rather than quantifying over the $T$-implications of $\varphi$ in $\mathfrak{A}_0$ and saying that they are \emph{all} consistent, we can let $\mathfrak{g}(\varphi)$ merely assert the conjunction according to which \emph{each} is consistent. That is: 

\begin{equation*}
\mathfrak{g}_0^\star(\varphi) :=  \left\{
        \begin{array}{ll}
            \bigwedge\{ \mathsf{Con}_T(\zeta) \mid \zeta\in \mathfrak{A}_0 \text{ and }T+\varphi\vdash \zeta \} & \quad \text{if $[\varphi]_T\neq[\bot]_T$}, \\
            \bot & \quad \text{otherwise}
        \end{array}
    \right.
\end{equation*}

In previous work we have shown that $\mathfrak{g}_0^\star$ oscillates between behaving like the identity and the consistency operator (see Theorem 3.7 in \cite{walsh2020note}). However, computing the function $\mathfrak{g}_0^\star$ requires access to the oracle $0'$. Indeed, to calculate $\mathfrak{g}_0^\star(\varphi)$ we must know whether $[\varphi]_T\neq[\bot]_T$, which requires $0'$. So $\mathfrak{g}_0^\star$ demonstrates that recursiveness is also a necessary condition in Theorem \ref{MC_for_con}; it cannot be weakened to limit-recursiveness.



\bibliographystyle{plain}
\bibliography{bibliography}

\end{document}